\documentclass{amsart}

\usepackage[backend=biber,style=alphabetic]{biblatex}
\usepackage[utf8]{inputenc}
\usepackage{amssymb}
\usepackage{amsthm}
\usepackage{amsmath}
\usepackage{quiver}
\usepackage{tikz} 
\usepackage{commath}
\usepackage{wasysym}
\usepackage{mathrsfs}
\usepackage[english]{babel}
\usepackage{tabularx}
\usepackage{tikz-cd}
\usepackage{multicol}
\usepackage[
backend=biber,
style=alphabetic,
]{biblatex}
\DeclareFieldFormat{labelalpha}{\mkbibbold{#1}}

\newtheorem{theorem}{Theorem}[section]
\newtheorem{proposition}[theorem]{Proposition}

\newtheorem{definition}[theorem]{Definition}

\newcommand{\R}{{\mathbb R}}

\newcommand{\T}{{\mathbb{T}}}

\newcommand{\inner}[1]{{\left\langle {#1} \right\rangle}}

\usepackage{geometry}
\title{Totally Geodesic Submanifolds in Products of Non-Positively Curved Manifolds}
\author{Nicholas Hanson}

\newcommand{\SFF}{{\textbf{II}_{\eta}}}

\newcommand{\PR}[1]{{\text{pr}_{#1}}}

\newtheorem{lemma}[theorem]{Lemma}
\newtheorem{corollary}[theorem]{Corollary}

\newtheorem{question}[theorem]{Question}

\begin{document}
\begin{abstract}
    We study non-positively curved closed manifolds $M$ and $n$-dimensional totally geodesic submanifolds of $M\times M$ which satisfy a transversality condition. We prove that, under some mild irreducibility requirements on $M$, if $M \times M$ admits infinitely many such submanifolds or just a single dense such submanifold, then $M$ is a locally symmetric space. In proving this, we prove a stronger version which only requires such submanifolds to exist in the universal cover $\widetilde M \times \widetilde M$. 
\end{abstract}

\maketitle

\section{Introduction}

\subsection{Background}

There is a general theme that if one can fit infinitely many special submanifolds into a larger manifold, then one can expect the larger manifold to have special properties. One particular example is the expectation that infinitely many maximal totally geodesic submanifolds in a negatively curved manifold will force the larger manifold to be locally symmetric. The goal of this paper is to prove a result in this direction for products of closed, non-positively curved manifolds.

Two recent examples of this phenomenon are the following. The first, a result of Fisher-Lafont-Miller-Stover is that for every $n \ge  3$, there exist infinitely many commensurability classes of finite-volume non-arithmetic hyperbolic n-manifolds for which the collection of all codimension one finite-volume totally geodesic submanifolds is finite but
nonempty \cite[Thm. 1.2]{FisherLaFont_2021}. This result was then quickly improved to obtain the following result of Bader-Fisher-Miller-Stover:

\begin{theorem}[{\cite[Thm. 1.1]{bader_arithmeticity_2020}}]\label{bfms}
    If $\Gamma < \text{SO}_0(n,1)$ is a lattice and the associated space $\text{SO}_0(n,1) / \Gamma$ contains infinitely many maximal totally geodesic subspaces of dimension at least $2$, then $\Gamma$ is an arithmetic subgroup.
\end{theorem}
This theorem can then be used to formulate some similar statements for totally geodesic submanifolds in hyperbolic manifolds. Another result was produced by Filip-Fisher-Lowe, who found that having infinitely many totally geodesic hypersurfaces in a negatively curved, real-analytic manifold yields a hyperbolic structure:

\begin{theorem}[{\cite[Thm. 3]{filip_finiteness_2024}}]\label{ffl}
    If $M$ is a closed, negatively curved, real-analytic Riemannian manifold of dimension at least $3$ which contains infinitely many totally geodesic immersed hypersurfaces, then it must be that $M$ is locally symmetric and hyperbolic.
\end{theorem}

Inspired by these results, Fisher asked the following general question (in personal communication):
\begin{question}\label{fisher-question}
    If $M$ is negatively curved and $M \times M$ contains infinitely many closed, totally geodesic submanifolds of dimension $n$ that do not lie entirely within the factors, must it be locally symmetric?
\end{question}

We prove the following positive result for the question:

\begin{theorem}\label{first-theorem}
    Let $M$ be a non-positively curved, closed and connected $n$-manifold (with  $n\ge 2$) which is topologically irreducible and admits no Euclidean factors. Assume there exists infinitely many closed, totally geodesic immersed $n$-dimensional submanifolds in the product $M\times M$ that are transverse to the projection maps. \\
    
    \noindent Then $M$ is locally symmetric. Furthermore, $M$ is an arithmetic manifold.
\end{theorem}

In answering this question, we prove the following, which is the Main Theorem of the paper and of which Theorem \ref{first-theorem} is a corollary: 
\begin{theorem}\label{main-theorem}
     Let $M = \widetilde M / \Gamma$ be a non-positively curved, closed and connected Riemannian manifold which is topologically irreducible (i.e. is not finitely covered by a product space) and which has no local Euclidean factors. Assume there exists an infinite number of $\Gamma$-distinct totally geodesic $n$-dimensional submanifolds of $\widetilde M \times \widetilde M$ which are transverse to the projection maps in $\widetilde M \times \widetilde M$.\\
     
     \noindent Then $M$ is locally symmetric.
\end{theorem}

Here, $\Gamma$-distinct means, loosely, that for two submanifolds in the universal cover, they are distinct from the perspective of the quotient $\widetilde M / \Gamma$. See Definition \ref{gamma-distinct} for a precise definition. Also note that the "transversality of factors" condition is necessary. Manifolds of the form $M\times \{p\} \subset M\times M$ are totally geodesic and any non-discrete Riemannian manifold $M$ admits infinitely many such submanifolds. There are further examples which are less trivial; see the discussion before Definition \ref{semidiagonal} for more information.

We will also prove a version of this theorem for dense totally geodesic submanifolds. Surprisingly, one only needs a single totally geodesic submanifold to prove that $M$ is locally symmetric:
\begin{theorem}\label{dense-theorem}
    Let $M$ be a non-positively curved, closed and connected Riemannian manifold of dimension $n$ which is topologically irreducible and contains no Euclidean factors. Assume there exists a dense, totally geodesic $n$-dimensional submanifold\\ $S\subset M \times M$.\\
    
    \noindent Then $M$ is locally symmetric.
\end{theorem}

If one only cares about negatively curved manifolds, then moving from non-positive curvature to negative curvature makes a slightly cleaner statement, since negatively curved, closed manifolds are topologically irreducible and cannot contain torus factors.
\begin{corollary} 
Let $M = \widetilde M / \Gamma$ be a negatively curved, closed and connected $n$-manifold. Then $M$ is locally symmetric in all of the following cases:
    \begin{enumerate} 
        \item If there are infinitely many closed, totally geodesic immersed $n$-dimensional submanifolds in the product $M\times M$ that are transverse to the projection maps. 

        \item If there exists an infinite number of $\Gamma$-distinct totally geodesic $n$-dimensional submanifolds of $\widetilde M \times \widetilde M$ which are transverse to the projection maps in $\widetilde M \times \widetilde M$. 

        \item If there exists a dense, totally geodesic $n$-dimensional submanifold $S\subset M \times M$ which is transverse to the projection maps.
    \end{enumerate}
    Furthermore, in case (1) $M$ is also an arithmetic manifold.
\end{corollary}

A more general version of the main theorem, which drops the hypotheses of topological irreducibility and allows for local Euclidean factors, is proven and discussed in Section \ref{general_case_section}.\\

\subsection{An Idea of the Proof}

    While Theorems \ref{first-theorem}, \ref{main-theorem}, and \ref{dense-theorem} have a similar flavor to the prior work mentioned in the introduction (Theorem \ref{bfms} and Theorem \ref{ffl}), there are two big differences between the theorems presented in this paper and previous work. First, the theorems in this paper apply to manifolds of non-positive curvature, whereas prior work focuses on strictly negative curvature. Secondly, the prior theorems make heavy use of dynamical techniques in their proofs. In Theorem \ref{ffl}, the proof looks at the leaves of the Anosov foliation generated by the geodesic flow on $M$, and in Theorem \ref{bfms} one studies invariant homogeneous measures on $\text{SO}_0(n,1)$. However, the proofs presented here use more geometric techniques and rely on studying the Levi-Civita connection and holonomy groups of the immersions $\widetilde S$ in $\widetilde M \times \widetilde M$.\\

The general idea for the proof of Theorem \ref{main-theorem} is as follows: for any given $\widetilde S \subset \widetilde M \times \widetilde M$, we study the composition of the inclusion map with a projection onto a factor of $\widetilde M$, which we will call $F$. We first show that $F$ is a diffeomorphism. This map preserves geodesics, and thus the holonomy group. We can then use Berger's classification of holonomy groups to determine that in fact $F$ is an isometry. We conclude that $\Gamma$ must have infinite index in the group of isometries $\text{Isom}(\widetilde M)$.\\

In order to upgrade the fact that $[\text{Isom}(\widetilde M) : \Gamma] = \infty$ to local symmetry, we require the following theorem of Farb and Weinberger, which creates a dichotomy within the category of aspherical manifolds: either the universal cover of an aspherical manifold $M$ has discrete isometry group, or $M$ may be characterized as an orbibundle.

\begin{theorem}[Farb-Weinberger, Theorem 1.2 \cite{farb_isometries_2008}]
    Let $M$ be a closed, aspherical manifold. Then either $\text{Isom}(\widetilde M)$ is discrete, or $M$ is isometric to an orbibundle
    $$F \rightarrow M \rightarrow B$$
    where
    \begin{itemize}
        \item $B$ is a good Riemannian orbifold, and $\text{Isom}(\widetilde M)$ is discrete.
        \item Each fiber $F$, with the induced metric, is a closed, aspherical, locally homogeneous Riemannian $n$-manifold, with $n > 0$. 
    \end{itemize}
\end{theorem}

Upon placing more conditions on our manifold $M$, one can force the base orbifold to be merely a point and force the fiber $F$ to be locally symmetric. This gives rise to the following theorem:

\begin{theorem}[Farb-Weinberger, Theorem 1.3 \cite{farb_isometries_2008} ]\label{farb-wein}
    Let $M$ be a closed Riemannian $n$-manifold with $n > 1$. Then the following are equivalent:
    \begin{itemize}
        \item $M$ is aspherical, topologically (smoothly) irreducible, $\pi_1(M)$ has no nontrivial normal abelian subgroups, and $\text{Isom}(\widetilde M)$ is not discrete.
        \item $M$ is isometric to a topologically (smoothly) irreducible, locally symmetric Riemannian manifold of non-positive curvature.
    \end{itemize}
\end{theorem}

We see then, after putting some restrictions on our Riemannian manifold, that if the universal cover admits a nontrivial convergent sequence of isometries then our universal cover must have been a symmetric space; that is, from a small amount of symmetry, we can infer that our manifold has a large amount of symmetry.\\

\noindent\textbf{Acknowledgments} I would like to thank my advisor Wouter van Limbeek, whose discussions and remarks were invaluable. I would also like to thank Benson Farb for his helpful edits for a cleaner exposition, and Ben Lowe for his comments and supplying a helpful example.

%
%
%
%
\section{Definitions and Preliminaries}

First, a quick rundown of some basic definitions.\\ 

We must first disambiguate two notions of irreducibility that we will encounter in this paper. First, we may speak of a Riemannian manifold $(M, g)$ being \emph{metrically irreducible} in the sense that it cannot be written as a product of two other Riemannian manifolds. That is, we cannot find $(N_1, g_1)$ and $(N_2, g_2)$ so that $M = N_1 \times N_2$ and $g = g_1 \oplus g_2$. The other notion of irreducibility we must reckon with is the one that appears in Theorem 1.2. We will say a manifold $M$ is $\emph{topologically (smoothly) irreducible}$ if there does not exist a non-trivial product manifold $N_1 \times N_2$ which finitely (and smoothly) covers $M$ (that is, we can take arbitrary finite covers without any of them being a product). Importantly, if a Riemannian manifold is topologically reducible, that product need not respect the metric structure.\\

Recall that a homothety $f$ between Riemannian manifolds $(M, g_M)$ and $(N, g_N)$ is a map $f: M \rightarrow N$ so that $f^*g_N = C \cdot g_M$ for some constant $C> 0$.\\

We will also recall some equivalent definitions of symmetric space:

\begin{definition}\label{symmetric_def} Let $X$ be a contractible Riemannian manifold. The following are equivalent:
\begin{itemize}
    \item $X$ is a symmetric space.
    \item $X$ is homogeneous (i.e. $\text{Isom}(X)$ acts transitively on $X$) and there is an isometric involution on $X$ which has an isolated fixed point.
    \item For any point $p\in X$, there is an isometric involution which has $p$ as an isolated fixed point.
\end{itemize}
\end{definition}

One can then say that $M$ is a locally symmetric space if $\widetilde M$ is a symmetric space.\\

To use the machinery from \cite{farb_isometries_2008}, it is also necessary to discuss warped products.
\begin{definition}
    If $F, B$ are manifolds, let $\mathcal{M}(F)$ be the moduli space of locally homogeneous metrics on $F$, and let $\lambda : B \rightarrow \mathcal{M}(F)$. Then the warped product $F \times_\lambda B$ is the manifold $F \times B$ with the metric at $T_{(x,b)}(F \times B)$ given by $g_{(x,b)}(v,w) = \lambda(b)(v,w) + g_B(v,w)$. \\

    For example, if $F$ is locally symmetric and $f : B \rightarrow \R_{>0}$, we may define the warped product $F \times_fB$ to have the metric $g_{x,b}^{F\times_F B} = f(b) \cdot g_x^F + g_b^B$.
\end{definition}

Now, we wish to study totally geodesic manifolds sitting inside a product $M \times M$, with the hope that having infinitely many such manifolds forces $M$ to have extra structure. However, we must be more discerning; given any closed $n$-manifold $M$, the product $M\times M$ will admit infinitely many totally geodesic $n$-submanifolds for the simple fact that $M\times \{q\}$ is a totally geodesic submanifold of $M\times M$. However, we cannot just disallow factors. We could have a non-locally symmetric $M$ with a totally geodesic hypersurface $N$; then if $\gamma$ is a closed geodesic, $S \times \gamma$ is a closed, totally geodesic submanifold in $M\times M$, and there exist infinitely many such submanifolds. To prevent these pathologies, we will force all our submanifolds to be ``semi-diagonal"; that is, we wish any submanifold $S$ we consider to be transverse to the factors $M\times \{q\}$ and $\{p\} \times M$. This leads us to the following definition, which generalizes the notion of the diagonal submanifold of $M\times M$:
\begin{definition}\label{semidiagonal}
    Let $S$ be an immersed, complete $n$-submanifold of $M\times M$. We will say $S$ is semi-diagonal if for $(p,q)\in S$, $S$ is transverse to the submanifolds $M \times \{q\}$ and $\{p\} \times M$. We may denote such a submanifold $(S, g_S, \iota_S)$ where $\iota_S$ is the immersion.
\end{definition}

In this paper, we will be concerning ourselves with semi-diagonal submanifolds that are also totally geodesic. Now, the obvious example of a family of semi-diagonal submanifolds of $M\times M$ are graphs of diffeomorphisms; in the case that we wish these semi-diagonal submanifolds to be totally geodesic, then the obvious example becomes the graph of an isometry (or, more generally, maps that preserve geodesics). If $M$ is negatively curved, we will see that it is actually the case that all such submanifolds arise from an isometry, up to taking a (possibly infinite) cover of $M \times M$.\\

For the main theorem, we require that there be infinitely many totally geodesic semi-diagonal submanifolds in the universal cover $\widetilde M \times \widetilde M$; however, given one such submanifold $\widetilde S$ which descends to a submanifold $S \subset M\times M$, one can act on $\widetilde S$ by the natural action of $\pi_1(M) \times \pi_1(M)$ to obtain distinct semi-diagonal totally geodesic submanifolds of $\widetilde M \times \widetilde M$, but yet they descend to the same submanifold $S$. Thus we must require some sense of distinctness in $\widetilde M \times \widetilde M$. This leads us to the following definition:

\begin{definition}\label{gamma-distinct}
    Let $F,G : S \rightarrow \widetilde M \times \widetilde M$ be immersions. We will say $F$ and $G$ are $\Gamma$-distinct if for any $\gamma_1, \gamma_2 \in \Gamma$ we have that $\gamma_2 \cdot \text{Im}(G\circ \gamma_1) \ne \text{Im}(F)$.
\end{definition}

Lastly, we will require that our manifold $M$ contains no local Euclidean factors. By the Flat Torus Theorem, this implies that the fundamental group contains no non-trivial normal abelian subgroups, thus satisfying the condition on the fundamental group of $M$ in Theorem \ref{farb-wein}. A treatment of the Flat Torus Theorem can be found in Chapter 7 of \cite{bridson_metric_2011}.

%
%
%
%
%
%
%
%

\section{Proof of the Main Theorem}

Now, our main theorem will be the following:
\begin{theorem}
    Let $M = \widetilde M / \Gamma$ be a non-positively curved, closed and connected Riemannian manifold which is topologically irreducible and which has no local Euclidean factors, and assume there exists an infinite number of $\Gamma$-distinct totally geodesic semi-diagonal $n$-submanifolds in $\widetilde M \times \widetilde M$. Then $M$ is locally symmetric.
\end{theorem}

But first, we will investigate how totally geodesic semi-diagonal manifolds arise as graphs of isometries. We first prove that they arise as diffeomorphisms up to a covering, and that they act nicely on geodesics.

\begin{lemma}\label{stretch_lemma}
Let $(S,g_S, \iota_S)$ is a totally geodesic semi-diagonal $n$-submanifold of $\widetilde M \times \widetilde M$, where $\widetilde M$ is non-positively curved. Then the map $F_i = \PR{i} \circ \iota : S \rightarrow \widetilde M$ is a diffeomorphism for $i=1,2$. Further, if $\gamma$ is a unit-speed geodesic in $S$, then $F_i(\gamma(t)) = \alpha(at)$, where $\alpha$ is a unit-speed geodesic in $\widetilde M$ and $a > 1$. 
\end{lemma}
\begin{proof}
    The fact that $\iota(S)$ is transverse to factors guarantees that $\iota$ is transverse to $\PR{i}$ as maps; so locally $\iota(S)$ is a graph in $\widetilde M \times \widetilde M$. Since $\iota$ is a local diffeomorphism from $S$ to $\iota(S)$, $F_i = \PR{i} \circ \iota$ is a local diffeomorphism.\\

    To prove injectivity of $F_i$, let $p,q \in \PR{i}^{-1}(x)$. Then there exists a geodesic $\gamma$ in $S$ between $p$ and $q$, which is also a geodesic in $\widetilde M \times \widetilde M$. Since $\PR{i}$ is a Riemannian submersion, $\PR{i}(\gamma)$ is a geodesic in $\widetilde M$ which intersects itself at the point $x$. Further, $\widetilde M$ is non-positively curved and simply connected; thus, the exponential map at $T_x\widetilde M$ is a diffeomorphism onto $\widetilde M$. So the only way for the geodesic $\PR{i}(\gamma)$ to intersect itself non-trivially at $x$ is for $\PR{i}(\gamma) = x$. So it must be that
    $\left(d_p\PR{i}\right)\gamma'(0) = 0$. But since $S$ is transverse to fibers, $\gamma'(0) = 0$ and $p=q$, so we can infer that $F$ is injective. Surjectivity of $F_i$ follows from the fact that $\iota\left(\widetilde S\right)\subset \widetilde M \times \widetilde M$ is complete, so that $\PR{i}\circ \iota\left(\widetilde S\right)$ is both open and closed in $M$.\\

    It is clear that $F$ preserves geodesics, since $\iota(\gamma(t))$ is a unit-speed geodesic in $\widetilde M \times \widetilde M$ and geodesics decompose in products, so that $\iota(\gamma(t)) = (\alpha(at), \beta(bt))$ with $\alpha,\beta$ unit-speed in $\widetilde M$ and $a,b > 1$ (where the inequality is strict due to transversality to factors).
\end{proof}

Next, we require a short proposition to take care of an edge case in the proof of the main theorem; however, the edge case is rather nice, it is the case in which our semi-diagonal submanifold is a symmetric space itself.\\

\begin{proposition}\label{symmetric_semi-diag}
    Let $(S, g_S, \iota)$ be a totally geodesic semi-diagonal $n$-submanifold of $\widetilde M \times \widetilde M$, where $M$ is non-positively curved. Then if $(S, g_S)$ is locally symmetric, $M$ is locally symmetric.
\end{proposition}
\begin{proof}
    Let $F = \PR{1} \circ \iota$. Then for an arbitrary point $p\in S$, let $q = F(p)$. We have an involutive isometry $s_p$ at $p$, and so we can obtain an involutive diffeomorphism $t_q = F^{-1} \circ s_p \circ F$ on $\widetilde M$. \\

    Now, choose $v \in T_x\widetilde M$ for some $x\in \widetilde M$. Then we can write $\exp_x v = \overline{\gamma}(\norm{v})$, for a unit-speed geodesic $\overline{\gamma}$. Since $F$ preserves geodesics by sending geodesics $\gamma(t)$ to geodesics $\overline{\gamma(ct)}$, we have $F^{-1}(x) = \text{exp}_{F^{-1}(x)}\left(c^{-1}\norm{v}\right)$ so that $d_xF(v) = c^{-1}v$. Then $d_{F^{-1}(x)}\circ d_xF(v) = c^{-1}v$ and $d_{s_p\circ F^{-1}(x)}F \circ d_{F^{-1}(x)}s_p\circ d_xF(v) = v$ by a similar argument. Thus, $t_p = F \circ s_p \circ F^{-1}$ preserves the norm induced by the Riemannian metric on the tangent bundle of $\widetilde M$; thus, it preserves the metric itself and is an isometry. Since such an involutive isometry exists at every point in $\widetilde M$, so $\widetilde M$ is a symmetric space. Thus, $M = \widetilde M / \Gamma$ is locally symmetric, finishing the proof.
\end{proof}

\begin{theorem}
    Let $M = \widetilde M / \Gamma$ be a non-positively curved, closed and connected Riemannian manifold which is topologically irreducible and which has no local Euclidean factors, and assume there exists an infinite number of $\Gamma$-distinct totally geodesic semi-diagonal $n$-submanifolds in $\widetilde M \times \widetilde M$. Then $M$ is locally symmetric.
\end{theorem}
\begin{proof}
    First, we begin by a series of reductions.\\ 
    
    We can automatically assume that all $S_k$ are non-symmetric, since if even one $S_k$ is a symmetric space then so too is $\widetilde M$ from Proposition \ref{symmetric_semi-diag}.\\


     We will also assume that each $S_k$ are metrically irreducible, then expand to the general case at the end.\\

    \noindent \textbf{Step I. F preserves parallel transport.} \\

    From Lemma \ref{stretch_lemma}, we know that $F$ sends geodesics to geodesics, changing the parameterization by a fixed factor of $\norm{d_pF(v)}$. We wish to show that $F$ reparameterizes each geodesic uniformly in every direction; that is, $\norm{d_pF(v)} = \norm{d_pF(w)}$ for all unit vectors $v,w \in T_p S$. To do so, we show that the parallel transport on $\widetilde M \times \widetilde M$ is left invariant by $F$. If $\beta$ is a path from $p$ to $q$, then we obtain an isometry $P_\beta^{\widetilde M \times \widetilde M} : T_p(\widetilde M \times \widetilde M) \rightarrow T_q(\widetilde M \times \widetilde M)$. Since the connection over a product manifold splits as $\nabla^{\widetilde M \times \widetilde M} = \nabla^{\widetilde M} \oplus \nabla^{\widetilde M}$, so too  does the parallel transport: $P_\beta^{\widetilde M \times \widetilde M} = P_\beta^{\widetilde M} \oplus P_\beta^{\widetilde M}$. Thus, we obtain the diagram
    
    \[\begin{tikzcd}
    	{T_p(\widetilde M \times \widetilde M)} & {} & {T_q(\widetilde M \times \widetilde M)} \\
    	{T_{\text{pr}_1(p)}\widetilde M} && {T_{\text{pr}_1(q)}\widetilde M}
    	\arrow["{P_\beta^{\widetilde M \times \widetilde M}}", from=1-1, to=1-3]
    	\arrow["{\text{pr}_1}"', from=1-1, to=2-1]
    	\arrow["{\text{pr}_1}", from=1-3, to=2-3]
    	\arrow["{P_\beta^{\widetilde M}}", from=2-1, to=2-3]
    \end{tikzcd}\]
    
    We wish to parallel transport vectors in $T_pS$ in order to relate $\norm{d_pF(v)}$ and $\norm{d_pF(w)}$. However, a priori we only know that parallel transport preserves $T_p(\widetilde M \times \widetilde M)$. Now, since $S$ is a totally geodesic submanifold of $\widetilde M \times \widetilde M$, the second fundamental form $\SFF(v) = \inner{B(v,v), \eta}$ vanishes for all $v\in T_pS, \eta \in T_p \left(\widetilde M \times \widetilde M\right)$. It follows that $P_\beta^{\widetilde M \times \widetilde M}(v) \in T_qS$, since if $V$ is the vector generated by parallel transport of $v$ along $\beta$, then $\nabla^{\widetilde M \times \widetilde M}_{\beta}V = \nabla_\beta^SV + B(V,V) = \nabla_\beta^S V$ the parallel transport of $v$ on $\widetilde M \times \widetilde M$ is identical to the parallel transport when restricted to $S$. Thus, we can restrict the domain and range of $P_\beta^{\widetilde M \times \widetilde M}$ to obtain
    
    \[\begin{tikzcd}
    	{T_pS} & {} & {T_qS} \\
    	{T_{\text{pr}_1(p)}\widetilde M} && {T_{\text{pr}_1(q)}\widetilde M}
    	\arrow["{P_\beta^{\widetilde M \times \widetilde M}}", from=1-1, to=1-3]
    	\arrow["{\text{pr}_1}"', from=1-1, to=2-1]
    	\arrow["{\text{pr}_1}", from=1-3, to=2-3]
    	\arrow["{P_\beta^{\widetilde M}}", from=2-1, to=2-3]
    \end{tikzcd}\]
    Now, it is easy to see that if we have $v,w\in T_pS$ with $\norm{v}_S = \norm{w}_S = 1$ with $P_\beta^{\widetilde M \times \widetilde M}v = w$, then 

    \begin{align*}
        \norm{dF_p(w)}_{\widetilde M} &= \norm{d(\text{pr}_1)_{f(p)}(w)}_{\widetilde M}\\
        &= \norm{d(\text{pr}_1)_{f(p)}(P_\beta^{\widetilde M \times \widetilde M}v)}_{\widetilde M}\\
        &= \norm{P_\beta^{\widetilde M} \circ dF_p(v)}_{\widetilde M}\\
        &=  \norm{dF_p(v)}_{\widetilde M}
    \end{align*}

    \noindent \textbf{Step II. Holonomy acts transitively on the sphere, so F acts uniformly on geodesics.} \\
    
    Now, let $C = \norm{dF_p(v)}_{\widetilde M}$. Then a corollary of Berger's classification of holonomy for metrically irreducible, simply connected manifolds is that either the holonomy group $\text{Hol}(S, g) \subseteq \text{SO}(n)$ acts transitively on the sphere, or $(S,g)$ is a symmetric space \cite{olmos_geometric_2005}. We have already assumed that $S$ is non-symmetric, so  we then have that there is some closed loop $\beta$ at $p$ so that $P^{\widetilde M \times \widetilde M}_\beta(v) = w$, and then $\norm{dF_p(v)}_{\widetilde M} = \norm{dF_p(w)}_{\widetilde M} = C$.\\
    
    Lastly, we note that given $p, q \in S$, there is a geodesic $\gamma$ from $p$ to $q$ with $\gamma(0) = p$ and $\gamma(L) = q$, by completeness. Then $\norm{d_pF(\gamma'(0))}_{S} = C$ so $F(\gamma(t)) = \gamma_1(Ct)$; thus, $d_qF(\gamma'(L)) = C$. Here $q$ is arbitrary, and thus $\norm{d_pF(v)}_{\widetilde M} = C$ for all $(p,v) \in T^1S$. Thus, we may conclude that for the metric $g_S$ for $S$, $F^*g_{\widetilde M} = C \cdot g_S$; that is, $F$ is a homothety from $S$ to $\widetilde M$.\\
    
    \noindent \textbf{Step III. F is an isometry.} \\
    
    Now, to prove $F$ is an isometry, the homothety gives an identification of \\
    $\left(\widetilde M,\text{pr}_1^*\left(g_{M\times M}|_{i(S)}\right)\right)$ (henceforth, this metric will be referred to as $g'$) with $(S, g_S) = (\widetilde M, C\cdot g_{\widetilde M})$. It follows that $F$ descends to a homothety between manifolds homeomorphic to $M$.\\
    
    Now, consider the set of all possible values of sectional curvature for $M$, the set $\text{Sec}(M) = \{\lambda_p(v \wedge w)\;|\; p\in S,\;\;v,w\in T_pM\}.$ Since $M$ is compact, we know that $\text{Sec}(M) = [-a, -b]$ for $a > b \ge 0$. Now, we must establish that $\text{Sec}(M\times M) = [-a, 0]$. If $v,w$ span a plane in $T_pS$, we may choose them to be orthonormal and decompose them into $v = v_1 + v_2$, $w = w_1 + w_2$ so that
    $$K_p(v \wedge w) = \frac{R(v,w,v,w)}{|v\wedge w|^2} = R(v_1,w_1,v_1,w_1) + R(v_2, w_2, v_2, w_2).$$

    This establishes that $K_p(v \wedge w) \le 0$. We further obtain that
    $$K_p(v \wedge w) \ge -a\left(|v_1 \wedge w_1|^2 + |v_2 \wedge w_2|^2\right).$$

    And we can see that 
    \begin{align*}
        |v \wedge w|^2 &\ge |v_1|^2|w_1|^2  + |v_2|^2|w_2|^2 \\
        &\ge |v_1|^2|w_1|^2 - \inner{v_1,w_1}^2 + |v_2|^2|w_2|^2 - \inner{v_2,w_2}^2 \\
        &=  |v_1 \wedge w_1|^2 + |v_2 \wedge w_2|^2
    \end{align*} 

    And so $\text{Sec}{(M\times M)} = [-a, 0]$. (Indeed both $-a$ and $0$ are obtained, $-a$ from a plane non-transverse to a factor and $0$ from a plane spanned by vectors lying in differing factors)
    
    Now, returning to our homothety we note that if $C \le 1$ then $C \cdot g_S = g'$, which implies that $[\frac{-a}{C^2}, \frac{-b}{C^2}] \subseteq [-a, 0]$. This then implies that $C$ must be $1$. Otherwise, we may consider $F^{-1}$, which is a homothety with constant $\frac{1}{C}$ and $C^2\cdot \text{Sec}(g') \subseteq [-a, -b]$, in which case yet again $C=1$. Thus, $F$ is in fact an isometry. \\

    \noindent \textbf{Step IV. Use Theorem \ref{farb-wein}.} \\

    We now have an infinite collection of isometries $f_i : \widetilde M \rightarrow \widetilde M$. Let $G = \text{Isom}(\widetilde M)$. Then we can consider the collection $[f_i] \in G / \Gamma$. Since $M / \Gamma$ is compact, $G / \Gamma$ is also compact. Further, since each $f_i$ are $\Gamma$-distinct, each $[f_i]$ are distinct. Thus $[f_i]$ admit a convergent subsequence which is not eventually constant, and $G = \text{Isom}(\widetilde M)$ is not a discrete group. Since $\Gamma$ contains no normal abelian subgroups and (by assumption) $M$ is not topologically irreducible, we can use Theorem \ref{farb-wein} to find that $\widetilde M$ is a symmetric space, finishing the proof in the metrically irreducible case.\\

    \noindent \textbf{Step V. The Case for Metrically Reducible $S$} \\

    Now, assume that $\widetilde S$ metrically decomposes into metrically irreducible Riemannian manifolds as $\widetilde S = \widetilde S_1 \times \cdots \times \widetilde S_N$. We still wish to show that $F : \widetilde S \rightarrow \widetilde M$ is an isometry. It is easy to see that the first step in the proof remains valid, so we can assume that $F$ preserves parallel transport. Thus, any isometry that may be induced by parallel transport on $\widetilde S$ may be induced on $\widetilde M$, and vice versa. It follows that $\text{Hol}(\widetilde M) = \text{Hol}(\widetilde S)$. We see that
    $$\text{Hol}(\widetilde S) = \text{Hol}(S_1) \times \cdots \times \text{Hol}(S_N).$$
    So $\text{Hol}(\widetilde M)$ decomposes in the same way. Since $\tilde M$ is simply connected and complete, the de Rham decomposition guarantees a decomposition of $\tilde M$ into factors $\tilde M_i$. It must be that $F$ sends $\widetilde S_i$ diffeomorphically onto a factor $\widetilde M_i$. Now we can use steps 2 and 3 to upgrade $F$ into an isometry on each factor. Thus, $F$ is an isometry. We can now use Theorem \ref{farb-wein} to obtain the desired result.
\end{proof}

\begin{corollary}
    Let $M$ be a closed, negatively curved Riemannian manifold, and assume there exists infinitely many totally geodesic semi-diagonal $n$-submanifolds in $M \times M$. Then $M$ is locally symmetric.
\end{corollary}
\begin{proof}
    We can lift the immersions $S$ to obtain the diagram
\begin{equation}
\begin{tikzcd}
	{\widetilde S} & {\widetilde M \times \widetilde M} & {\widetilde M} \\
	S & {M \times M} & M
	\arrow["{\tilde\iota}", from=1-1, to=1-2]
	\arrow["F"{description}, curve={height=-24pt}, from=1-1, to=1-3]
	\arrow[from=1-1, to=2-1]
	\arrow["{\text{pr}_1}", from=1-2, to=1-3]
	\arrow["{p \times p}"{description}, from=1-2, to=2-2]
	\arrow["p", from=1-3, to=2-3]
	\arrow["\iota", from=2-1, to=2-2]
	\arrow["{\text{pr}_1}", from=2-2, to=2-3]
\end{tikzcd}
\end{equation}

Then we can apply Theorem \ref{main-theorem} to obtain that $M$ is locally symmetric.

To further prove that $M$ is arithmetic, we note that since $M/\Gamma$ is a closed manifold, it must be that $F$ commensurates $\Gamma$. Since the $F$ are $\Gamma$-distinct, it must be that $[\text{Comm}_G(\Gamma) : \Gamma] = \infty$. This implies $\Gamma$ is dense in $\text{Comm}_G(\Gamma)$ (Prop. 6.2.3 \cite{Zimmer1984ErgodicTA}) and so by Margulis Arithmeticity $\Gamma$ is an arithmetic lattice.
\end{proof}

\begin{theorem}
    Let $M$ be a negatively curved, closed and connected Riemannian manifold which is topologically irreducible and contains no local Euclidean factors. Assume there exists a dense, totally geodesic $n$-submanifold $S\subset M \times M$. Then $M$ is locally symmetric.
\end{theorem}
\begin{proof}
    We can lift the map $\iota$ to the universal cover $\widetilde S$ to obtain the diagram (1) from the proof of Corollary \ref{first-theorem}. From the proof of Theorem \ref{main-theorem}, we can assume $F$ is an isometry, i.e. $F\in \text{Isom}(\widetilde M)$. Now, let $H = \inner{\Gamma, F} \le \text{Isom}(\widetilde M)$. Now, if it were the case that $[H : \Gamma] < \infty$, then the volume of $S \subset M \times M$ would be a finite multiple of the volume of $M$. However, if that were the case the image $\iota(S)$ would be compact. Since $\iota(S)$ is dense, it must be that $[H : \Gamma] = \infty$. Thus, $\text{Isom}(\widetilde M)$  is not discrete and so, according to Theorem \ref{farb-wein}, $\widetilde M$ is a symmetric space. 
\end{proof}

\section{The General Case}\label{general_case_section}

\subsection{Eliminating Irreducibility Requirements}

In the statement of the main theorem, we require two assumptions. Firstly, the manifold $M$ is topologically irreducible and secondly $M$ contains no local Euclidean factors. Both conditions can be relaxed, at a cost.

From the proof of Theorem \ref{main-theorem}, we see that all semi-diagonal submanifolds on a non-positively curved manifold $M$ arise from isometries of the universal cover $\widetilde M$. On the other hand, a product of tori $\T^n \times \T^n$ contains an infinite number of semi-diagonal submanifolds which do not arise as isometries of the universal cover. In particular, for any linear transformation $A\in \text{GL}_n(\R)$, the submanifold $\text{Graph}(A) \subset \R^n \times \R^n$ descends to a semi-diagonal submanifold in $\T^n\times \T^n$. So any torus factor immediately gives an infinite number of semi-diagonal manifolds. 

For the condition of topologically irreducibility, \cite{farb_isometries_2008} contains a discussion of the more general case. In particular, Proposition 3.1 of \cite{farb_isometries_2008} claims that in the case $\text{Isom}(M)$ is semisimple (i.e. has semisimple Lie algebra) with finite center, then $M$ is finitely covered by a warped product $F \times_f B$, where $F$ is a nontrivial, nonpositively curved locally symmetric space. Further, Proposition 3.3 of \cite{farb_isometries_2008} claims that if $M$ contains no local Euclidean factors then $\text{Isom}(\widetilde M)$ is indeed semisimple with finite center. From this discussion, we see that we may completely relax the conditions on $M$, decompose $M = \T^n \times N$ and apply Proposition 3.1 of \cite{farb_isometries_2008} to $N$ to obtain a complete classification. 

The general idea is to simply split off the troublesome flat factor. However, in order to then determine the warped product structure on the negatively curved factor, we must be able to restrict both the codomain and domain of the maps induced by the totally geodesic semi-diagonal submanifolds. This requires proving that the maps themselves split. That is, if $F : \R^n \times \widetilde N \rightarrow \R^n \times \widetilde N$ is an induced map then we can write $F = (F_1, F_2)$ for $F_1 : \R^n \rightarrow \R^n$ and $F_2 : \widetilde N \rightarrow \widetilde N$.

\begin{lemma}\label{seperation_lemma}
    If $\widetilde S$ is a totally geodesic semi-diagonal $n$-submanifold of $\widetilde M \times \widetilde M$, with $\widetilde M = \R^n \times \widetilde N$, and $F : \widetilde S \rightarrow \R^n \times \widetilde N$ is the induced map, then $\widetilde S$ is isometric to $\R^n \times \widetilde N$ and we can write $F(p, q) = ((F_1(p), F_2(q))$ where $F_1 \in \text{Aff}(\R^n)$ and $F_2 \in \text{Isom}(\widetilde N)$.
\end{lemma}
\begin{proof}

The proof is very similar to Step 5 of Theorem \ref{main-theorem}. The map $F = \text{pr}_2 \circ \iota : \tilde S \rightarrow \R^n\times \tilde N$ is still a diffeomorphism and preserves parallel transport, so we may identify $\tilde S$ diffeomorphically with $\R^n \times \tilde N$ and $F$ becomes a smooth map from $\R^n \times \widetilde N$ to itself. Now, $F^{-1}$ also preserves parallel transport and is an equivariant map with respect to the holonomy representation of $\R^n \times \widetilde N$. Thus, it must send the flat factor $\R^n$ to itself and it must send the non-positively curved factor $\tilde N$ to itself. So we can write $F(p, q) = (F_1(p), F_2(q))$, where $F_1 : \R^n \rightarrow \R^n$ and $F_2 : \widetilde N \rightarrow \widetilde N$.


Finally, for $F_2$ we know from the proof of \ref{main-theorem} to see that $F_2 \in \text{Isom}(\widetilde N)$. For $F_1$, we know it must preserve geodesics and so $F_1 \in \text{Aff}(\R^n)$ by the Fundamental Theorem of Affine Geometry \cite{scherk_affine_theorem_1962}.
\end{proof}

\begin{theorem}
    Let $M$ be a non-positively curved, closed and connected Riemannian manifold and assume that $\widetilde M\times \widetilde M$ contains infinitely many totally geodesic, semi-diagonal $n$-submanifolds which are $\pi_1(M)$-distinct. Then $M$ admits a finite cover which can be written as 
    $$\T^n \times (X \times_f B),$$
    such that $B \times B$ contains only finitely many semi-diagonal totally geodesic submanifolds and $X$ is a (possibly trivial) locally symmetric space.
\end{theorem}
\begin{proof}
    Apply Lemma \ref{seperation_lemma} to obtain $\widetilde M \cong \R^n \times \widetilde N$ so that $\T^n \times N$ finitely covers $M$. Then, as in the proof of the main theorem, we may obtain maps $F : \R^n \times \widetilde N \rightarrow \R^n \times \widetilde N$. We can restrict these maps in both domain and range to get an infinite sequence of maps corresponding to $\pi_1(M)$ distinct semi-diagonal submanifolds in $\widetilde N \times \widetilde N$. In the case that after projection there exists only a finite number of distinct maps, then $B = \widetilde N$ and $F = \{*\}$. Otherwise, the isometry group of $N$ is semisimple and non-discrete, so by Proposition 3.1 of \cite{farb_isometries_2008} it admits a finite cover which is a warped product $F \times_f B$, where $F$ is a nontrivial symmetric space.
\end{proof}

\subsection{Generalizing to Arbitrary Products}

Throughout the proof of Theorem \ref{main-theorem}, we note that we choose one factor to project down, and the composition of the projection and inclusion map yields an isometry $F : \widetilde S \rightarrow \widetilde M$. However, we do not use the fact that the product $\widetilde M \times \widetilde M$ has isometric factors. Thus, we can make a slightly more general statement by relaxing this hypothesis, and looking at submanifolds inside products $M_1 \times M_2$. 

\begin{theorem}
    For $i=1,2$, let $M_i$ are non-positively curved closed Riemannian manifolds with $M_i = \widetilde M_i / \Gamma_i$. If $M_1 \times M_2$ contains infinitely many totally geodesic, semi-diagonal $n$-submanifolds, then $\widetilde M_1 \cong \widetilde M_2$, both are locally symmetric, and $\Gamma_1$ and $\Gamma_2$ are commensurable lattices.
\end{theorem}
\begin{proof}
    It is clear that if we have such manifolds, then from the proof of Theorem \ref{main-theorem} it is immediate that $\widetilde M_1 \cong \widetilde M_2$ and that $M_1$ and $M_2$ are locally symmetric. Lastly, given any totally geodesic, semi-diagonal submanifold $S$ in $M_1 \times M_2$, $S$ finitely covers both $M_1$ and $M_2$ via the projection maps. So $M_1$ and $M_2$ share a common finite cover, and therefore the corresponding lattices $\Gamma_1$ and $\Gamma_2$ are commensurable.
\end{proof}

Another easy generalization in the same direction is to look at more than two products of manifolds, i.e. looking at $n$-dimensional totally geodesic semi-diagonal submanifolds of the $kn$ dimensional manifold $M_1 \times \cdots \times M_k$. Luckily the notion of semi-diagonality continues to apply here; it still makes sense to ask for a submanifold to be transverse to the projection maps $\text{pr}_i$. One may then run through the argument from Theorem \ref{main-theorem} after projecting onto a factor in order to get a quick, slight improvement of the main theorem.

\begin{corollary}\label{arbitraryproductcorollary}
    Let $M_i = \widetilde M_i / \Gamma_i$, for $i = 1,\cdots , k$, be non-positively curved, closed and connected Riemannian $n$-manifolds which are topologically irreducible and have no Euclidean factors. Assume there exists an infinite number of $\left(\Gamma_1 \times \cdots \times \Gamma_k\right)$-distinct totally geodesic, semi-diagonal $n$-submanifolds in $\widetilde M_1 \times \cdots \times \widetilde M_k$. Then the universal covers $\widetilde M_i$ are all isometric, each manifold $M_i$ is locally symmetric, and all lattices $\Gamma_i$ are commensurable.
\end{corollary}

\section{Other Questions}

In everything preceding this section, we assume that our totally geodesic, semi-diagonal submanifolds $S$ have the same dimension as the factors of the product inside of which it sits, and this assumption is important in most of the proofs. One might then be interested in the case of submanifolds of varying dimensions. That is, we could study manifolds of dimension, say, $n-1$ or $n+1$ inside $M\times M$. One expects that lowering the dimension would give us less information and thus we would have a weaker result (or potentially no result at all), and raising the dimension would lead to a similar (or perhaps even stronger) result.

\begin{question}
    If $M, N$ are closed, topologically irreducible non-positively curved Riemannian manifolds and $M \times N$ contains infinitely many totally geodesic semi-diagonal $k$-submanifolds, for $n < k < 2n$, are $M$ and $N$ locally symmetric?
\end{question}

For the lowering of dimension, there is a straightforward counterexample to Corollary \ref{arbitraryproductcorollary} in the case of submanifolds of dimension $n-2$. First, we can consider an arithmetic manifold; for example, we can let $G = \text{SO}_{n,1}(\R)$ for $n> 2$ and $\Gamma < G$ be a cocompact, arithmetic lattice; then $\Gamma \backslash G / K$ is a closed arithmetic manifold. Now let $S \le G$ be $\text{SO}_{n-1,1}(\R)$; since $\Gamma$ is arithmetic, $\text{Comm}_G(\Gamma)$ is infinite by the Margulis theorem. The cosets of $S$ under the quotient map $G/K \rightarrow \Gamma\backslash G/K$ give rise to infinitely many distinct immersed submanifolds $[xS] \subseteq \Gamma\backslash G / K$. Next, we can find a manifold $N$ which is not locally symmetric, not a warped product, and contains a hyperbolic $(n-1)$-plane $N' \subseteq N$; say, for example, by perturbing the metric on a hyperbolic manifold away from a hyperbolic hypersurface. Then for each hyperbolic space $xS$, we have an isometric identification $f_x : xS \rightarrow \widetilde N'$, where $\widetilde N'$ is the lift of $N'$ to $\widetilde N$. Thus, the product $\widetilde M \times \widetilde N$ contains infinitely many totally geodesic $(n-2)$-submanifolds which are transverse to factors, yet $N$ is not locally symmetric (or even a warped product). Of course, this is a simple example but it immediately leads to the more interesting question:

\begin{question}
    Do there exist non-locally symmetric, topologically irreducible non-positively curved closed manifolds $\widetilde M$, $\widetilde N$ which contain infinitely many totally geodesic semi-diagonal $k$ submanifolds for some $1 < k < n$?
\end{question}

Lastly, a reasonable extension of this theorem would be to prove a version for positively curved manifolds or lift the curvature restriction altogether. However, the proof of Theorem \ref{main-theorem} is doomed to fail. Firstly, Theorem \ref{farb-wein} of Farb and Weinberger requires asphericality of the universal cover, and so does not apply to any positively curved symmetric spaces. One may still hope that, at the very least, semi-diagonal submanifolds of $M \times M$ or $\widetilde M \times \widetilde M$ still arise from isometries. However, this fact is also in jeopardy; the first statement in Lemma \ref{stretch_lemma} requires the fact that the exponential map is a diffeomorphism on a non-positively curved manifold in order to prove that $F = \text{pr}_i \circ \iota$ is a diffeomorphism. We can conclude that any similar statement would use very different tools than those used in this paper.

\newpage

\printbibliography

\end{document}